\documentclass[12pt, twoside]{article}
\usepackage{amsmath,amsthm,amssymb}
\usepackage{times}
\usepackage{enumerate}

\def\subjclass#1{{\renewcommand{\thefootnote}{}%
\footnote{\hspace{-0.6cm}\emph{Mathematics Subject Classification
(2020):} #1}}}
\def\subj#1{{\renewcommand{\thefootnote}{}%
\footnote{\hspace{-0.6cm}\emph{keywords:} #1}}}

\newtheorem{thm}{Theorem}[section]
\newtheorem{cor}[thm]{Corollary}
\newtheorem{lem}[thm]{Lemma}

\newtheorem{mainthm}[thm]{Main Theorem}

\theoremstyle{definition}
\newtheorem{defin}[thm]{Definition}
\newtheorem{rem}[thm]{Remark}
\newtheorem{exa}[thm]{Example}

\numberwithin{equation}{section}

\begin{document}


\baselineskip=17pt


\title{\bf\textsc{Periodicity of $p$-adic Expansion
of Rational Number}}

\author{Rafik BELHADEF, Henri-Alex ESBELIN}

\date{}

\maketitle

 \subj{$p$-adic expansion, $p$-adic
number, rational number}


\begin{abstract}
In this paper we give an algorithm to calculate the coefficients of
the $p$-adic expansion of a rational numbers, and we give a method to decide whether this expansion is periodic or ultimately periodic.

\end{abstract}

\section{Introduction}
It is known that in $\mathbb{R}$, an element is rational if and only
if its decimal expansion is ultimately periodic. An important
analogous theorem for the $p$-adic expansion of rational
number, is given by the following statement (see \cite{1}):
\begin{thm}
The number $x\in \mathbb{Q}_{p}$ is rational if
and only if the sequence of digits of its $p$-adic expansion is periodic or ultimately periodic.
\end{thm}

For example, in $\mathbb{Q}_{3}$, the $p$-adic expansion of
$-\frac{1}{2}$ is $1+3+3^{2}+3^{3}+...=111111111111 $,
it is clear that this expansion is purely periodic. In the second example in $\mathbb{Q}_{3}$, the $p$-adic expansion of
$\frac{11}{5}$ is given by $1+1.3+1.3^{2}+2.3^{3}+1.3^{4}+0.3^{5}+...=1112101210121012101210.....$
This expansion is ultimately periodic, with periodic block $1210$. Another example in $\mathbb{Q}_{5}$, the $p$-adic expansion of
$\frac{213}{7}$ is given by $4+1.5+3.5^{2}+1.5^{3}+4.5^{4}+2.5^{5}+3.5^{6}+0.5^{7}+2.5^{8}+...=413142302142302...$
This expansion is ultimately periodic, with periodic block $142302$. 

Evertse in \cite{3}, gave an algorithm to calculate the coefficients of $p$-adic expansion of an element in
$\mathbb{Z}_{p}$. We continue the study of the characterization of p-adic numbers (see \cite{2}), we inspired by the works of Evertse, we propose the algorithm (\ref{frac}), to calculate a sequence of digits of a rational number $\dfrac{c}{d}$, then we prove that this sequence defines the $p$-adic expansion of $\dfrac{c}{d}$ (see lemma \ref{coef-hens}), and verified a relationship (\ref{relatio}) (see lemma \ref{dev-ratio}). Finally, in the main theorem, we demonstrate the periodicity of the $p$-adic expansion of $\dfrac{c}{d}$.

\section{Definitions and properties}
We will recall some definitions and basic facts from $p$-adic
numbers (see \cite{4}. Throughout this paper $p$ is a prime number, $\mathbb{Q}$
is the field of rational numbers, $\mathbb{Q}^+$ is the field of
nonnegative rational numbers and $\mathbb{R}$ is the field of real
numbers. We use $\left|.\right|$ to denote the ordinary absolute
value, $v_p$ the $p$-adic valuation and $\left|.\right|_p$ the $p$-adic
absolute value. The field of $p$-adic numbers $\mathbb{Q}_{p}$ is
the completion of $\mathbb{Q}$ with respect to the $p$-adic absolute
value. We denote the ring of $p$-adic integers by $\mathbb{Z}_{p}$.
Every element of $\mathbb{Q}_{p}$ can be expressed uniquely by the $p$-adic expansion $\overset{+\infty }{
\underset{n=-j}{\sum }}\alpha _{n}p^{n}$ with $\alpha _{i}\in
\{0,1,..,p-1\}$ for $i\geq -j$. In $\mathbb{Z}_{p}$ we have simply
$j=0$.\\

Now, we give in the following definition the requested algorithm for a rational number

\begin{defin} \label{algo-hensel}
\label{Algo-ratio}Let $\dfrac{c}{d}\in \mathbb{Q} ^{+}\cap \mathbb{Z}_{p}$ , with
$c\in \mathbb{N} $ , $d\in \mathbb{N}^{\ast }$, and $(c,p)=1$, $(d,p)=1$, $(c,d)=1$. We define the sequences $\left( \alpha _{i}\right) _{i\in\mathbb{N}}$ and $\left( \beta _{i}\right) _{i\in\mathbb{N}}$ by

\begin{equation}\label{frac}
\left\{
\begin{array}{l} 
\beta _{0}=c \\
\\
\alpha _{i}=\beta _{i}d^{-1}\text{mod}p , \forall i\geq 0 \\
\\
\beta _{i+1}=\dfrac{\beta _{i}-\alpha _{i}d}{p}\in\mathbb{Z} ,  \forall i\geq 0 
\end{array}
\right. 
\end{equation}
\end{defin}

\begin{lem}
\label{coef-hens} Under the hypothesis of the definition (\ref{Algo-ratio}), the $p$-adic expansion of $\dfrac{c}{d}$ is given by
$\overset{+\infty }{%
\underset{i=0}{\sum }}\alpha _{i}p^{i}$, with $\alpha _{i}\in
\{0,1,..,p-1\}$, $\forall i\geq 0 $. The opposite is true, i.e, if $\dfrac{c}{d}=\overset{+\infty }{
\underset{i=0}{\sum }}\alpha _{i}p^{i}$, then the sequences $\left( \alpha _{i}\right) _{i\in\mathbb{N}}$ and $\left( \beta _{i}\right) _{i\in\mathbb{N}}$ verified the algorithm (\ref{frac}).
\end{lem}

\begin{proof}
Let $\left( \alpha _{i}\right) _{i\in\mathbb{N}}$ and $\left( \beta _{i}\right) _{i\in\mathbb{N}}$ as in the definition (\ref{Algo-ratio}). We have 
\begin{eqnarray*}
\dfrac{c}{d}&=& \alpha _{0}+\dfrac{\beta_{1}}{d}p\\
&=& \alpha _{0}+\alpha _{1}p+\dfrac{\beta_{2}}{d}p^2\\
&...&\\
&=& \alpha _{0}+\alpha _{1}p+...+\alpha _{n}p^n+\dfrac{\beta_{n+1}}{d}p^{n+1}
\end{eqnarray*}
So $$\left| \dfrac{c}{d}-\overset{n}{
\underset{i=0}{\sum }}\alpha _{i}p^{i} \right|_p \leq \dfrac{1}{p^{n+1}}$$
therefore $\overset{+\infty }{
\underset{i=0}{\sum }}\alpha _{i}p^{i} = \dfrac{c}{d}$.

For the second part, we suppose $\dfrac{c}{d}=\overset{+\infty }{
\underset{i=0}{\sum }}\alpha _{i}p^{i}$, and we prove by recursion that the sequences $\left( \alpha _{i}\right) _{i\in\mathbb{N}}$ and $\left( \beta _{i}\right) _{i\in\mathbb{N}}$ verified the algorithm (\ref{frac}). For $i=0$, we have
$\dfrac{c}{d}=\alpha _{0}\text{mod}p$, then $\alpha
_{0}=cd^{-1}\text{mod}p=\beta _{0}d^{-1}\text{mod}p$. 
Now, suppose that $\alpha _{i}=\beta _{i}d^{-1}\text{mod}p$ \ and $\beta
_{i+1}=\dfrac{\beta
_{i}-\alpha _{i}d}{p}$, so we have
\begin{eqnarray*}
\alpha _{i} &=&\beta _{i}d^{-1}\text{mod}p\Longrightarrow \alpha
_{i+1}p+\alpha _{i}=\beta _{i}d^{-1}\text{mod}p \\
\\
&\Longrightarrow &\alpha _{i+1}p=\left( \beta _{i}d^{-1}-\alpha
_{i}\right)
\text{mod}p  \\
\\
&\Longrightarrow &\alpha _{i+1}=\left( \frac{\beta _{i}-\alpha _{i}}{p}%
\right) d^{-1}\text{mod}p=\beta _{i+1}d^{-1}\text{mod}p
\end{eqnarray*}
therefore $\forall i\geq 0:\alpha _{i}=\beta
_{i}d^{-1}\text{mod}p.$
\end{proof}

\begin{lem}
\label{dev-ratio} Under the hypothesis of the definition (\ref{Algo-ratio}),\ we have
\begin{equation}\label{relatio}
c=d\left( \overset{i-1}{\underset{n=0}{\sum }}\alpha
_{n}p^{n}\right) +\beta _{i}p^{i}\text{ \ \ \ \ , \ \ }\forall
i\in\mathbb{N} ^{\ast }
\end{equation}
\end{lem}

\begin{proof}
We prove this lemma, also, by induction. For $i=1$, it's obvious.
\begin{equation*}
d\left( \overset{0}{\underset{n=0}{\sum }}\alpha _{n}p^{n}\right)
+\beta _{1}p=d\alpha _{0}+\left( \frac{c-\alpha _{0}d}{p}\right) p=c
\end{equation*}%
Suppose that, the relationship is true for $i$. From
(\ref{frac}), we have $\beta _{i}=\alpha _{i}d+\beta _{i+1}p$.
Then
\begin{eqnarray*}
c &=&d\left( \overset{i-1}{\underset{n=0}{\sum }}\alpha
_{n}p^{n}\right)
+\beta _{i}p^{i} \\
&=&d\left( \overset{i-1}{\underset{n=0}{\sum }}\alpha
_{n}p^{n}\right)
+\left( \beta _{i+1}p+\alpha _{i}d\right) p^{i} \\
&=&d\left( \overset{i}{\underset{n=0}{\sum }}\alpha _{n}p^{n}\right)
+\beta _{i+1}p^{i+1}
\end{eqnarray*}
So, the relationship is true for all $i\in \mathbb{N}$.
\end{proof}
\begin{rem}
Let $r=\dfrac{c'}{d'}\in \mathbb{Q} ^{+}$, but not in $\mathbb{Z}_{p}$, i.e. the $p$-adic expansion of
 $\dfrac{c'}{d'}$ is given by $\overset{+\infty }{%
\underset{n=-j}{\sum }}\alpha _{n+j}p^{n}$, with $j \neq 0$ and  $\alpha _{i}\in
\{0,1,..,p-1\}$ ,\ \ $\forall i \geq -j$. In this case, we can suppose
$c'=c\in \mathbb{N} $ , $d'=p^j d\in \mathbb{N}^{\ast }$, with $(d,p)=1$, and $(c,p)=1$. So, we have $\dfrac{c}{d}=\overset{+\infty }{
\underset{n=0}{\sum }}\alpha _{n}p^{n}$. We define a sequence $\left( \beta _{i}\right) _{i\in\mathbb{N}}$ by the same way
\begin{equation}
\left\{
\begin{array}{l}
\beta _{0}=c=c' \\
\\
\beta _{i+1}=\dfrac{\beta _{i}-\alpha _{i}d}{p} = \dfrac{\beta _{i}p^j-\alpha _{i}d'}{p^{j+1}} \in\mathbb{Z}
\end{array}%
\right. 
\end{equation}

\end{rem}

\section{Results and proof}

To show that the algorithm (\ref{Algo-ratio}) stops after a certain rank, it suffices to prove that the sequence $\left(
\left\vert \beta _{n}\right\vert \right) _{n\in\mathbb{N}}$ is bounded or decreasing. This is the subject of the main theorem.

\begin{mainthm}
The sequence $\left( \beta _{i}\right) _{i\in \mathbb{N}}$ given in (\ref{frac})\ verified the following cases:
\\
\textbf{Case1. }If $c<d$, then
\begin{equation*}
0\leq \left\vert \beta _{i}\right\vert <d\text{ \ \ , \ }\forall
i\in \mathbb{N}
\end{equation*}%
\textbf{Case2. }If $c>d$ and $p\geq 3$, we have, also, two
cases:\\

    \textbf{Case2.1. }If $0<\frac{c(p-1)}{2dp}<1$, then for all $i\in\mathbb{N} ^{\ast }$, we have
$\left\vert \beta _{i}\right\vert <d$.
\\

\textbf{Case2.2. }
 If $1<\frac{c(p-1)}{2dp}$, then for a fixed integer
\begin{equation}\label{fixed-integer}
m=\left[ \dfrac{\log \left( \dfrac{c(p-1)}{2dp}\right) }{\log
p}\right]
\end{equation}
it comes that
\begin{equation*}
\left\{
\begin{array}{l}
d < \left\vert \beta _{i}\right\vert < c\text{ \ \ \ \ for \ \ \ }%
0\leq i<m+1 \\
\\
0\leq \left\vert \beta _{i}\right\vert < d\text{ \ \ \ \ for\ \
\ \ \ \ \ \
\ \ }m+1 < i \\
\\
0\leq\left\vert \beta _{i}\right\vert < c\text{ \ \ \ \ for\ \
\ \ \ \ \ \
\ \ }m+1 = i%
\end{array}%
\right.
\end{equation*}
\end{mainthm}

\begin{proof}
We treat all cases:\\
Case1.  Let $c<d$, we use the proof by induction. For $i=0 $ is trivial. We suppose that in the rank $n$ we have $\left\vert \beta _{i}\right\vert <d$, and we prove the inequality
$\left\vert \beta _{i+1}\right\vert <d$ .\ Indeed, we have%
\begin{eqnarray*}
\left\vert \beta _{i+1}\right\vert &=&\left\vert \dfrac{\beta
_{i}-\alpha
_{i}d}{p}\right\vert \\
&<&\frac{1}{p}\left\vert \beta _{i}\right\vert +\frac{1}{p
}\left\vert \alpha _{i}d\right\vert \\
&<&\frac{1}{p}d+\frac{p-1}{p}d=d
\end{eqnarray*}
Case2. For $c>d$ and $p\geq 3$, we prove the two following cases:\\ 

Case2.1. We suppose $0<\frac{c(p-1)}{2dp}<1$. Also, we prove by recurrence that $\left\vert \beta _{i}\right\vert <d$. Starting with $i=1$, we have
\begin{eqnarray*}
0<\frac{c(p-1)}{2dp}<1\Longleftrightarrow -\dfrac{\alpha_{0}d}{p}<\frac{c}{p}-%
\dfrac{\alpha_{0}d}{p}<\frac{2d}{p-1}-\dfrac{\alpha_{0}d}{p}
\end{eqnarray*}
So
\begin{eqnarray*}
-d<-\dfrac{\alpha_{0}d}{p}<\beta _{1}<d\left( \frac{2}{p-1}-%
\dfrac{\alpha_{0}}{p}\right) <d
\end{eqnarray*}
Now, we assume that the property is true at rank $i$, and we show it
at rank $i+1$. Indeed, we have
\[
-d<\beta _{i}<d\Longleftrightarrow -d<\frac{-d\left(
1+\alpha_{i}\right) }{p}<\dfrac{\beta
_{i}-\alpha_{i}d}{p}<\frac{d\left( 1-\alpha_{i}\right) }{p}<d
\]
then $-d<\beta _{i+1}<d$. Which means that for every $i\in\mathbb{N}
^{\ast }$, we have $\left\vert \beta _{i}\right\vert <d$.

Case2.2. Let the integer $m$ given in (\ref{fixed-integer}), we suppose that $1<\dfrac{c(p-1)}{2dp}$.\\
Firstly, we will prove that for all $0\leq i\leq m$ the terms
$\beta _{i}$\ are strictly positive. Indeed, we assume that there is $k\in \left\{ 1,...,m\right\} $, such
that $\beta _{k}<0$. From definition (\ref{Algo-ratio}),
we have $$\dfrac{\beta _{k-1}-\alpha_{k-1}d}{p} <0$$
which means $\beta _{k-1}<dp$. Multiplying both sides by $p^{k-1}$, and applying the lemma (\ref{dev-ratio}), it comes $$c<d\left( \overset{k-2}{\underset{n=0}{\sum }}
\alpha_{n}p^{n}\right) +dp^{k}$$
The coefficients $\alpha_{n}$ are strictly less than $p$, so
$$c<dp\left( \dfrac{p^{k-1}-1}{p-1}+p^{k-1}\right) $$
Then, after simplification
\begin{eqnarray*}
c<\frac{pd}{p-1}\left( p^{k}-1\right) <\frac{2pd}{p-1}
p^{k}
\end{eqnarray*}
Thus
\[\dfrac{\log \left( \dfrac{c(p-1)}{2dp}\right) }{\log p}<k\]
however $m+1\leq k$. Where does the contradiction come from. Which means
that for every $0\leq i\leq m$, we have $\beta _{k}>0$.\\
Now, we prove the inequalities $d\leq \beta
_{i}\leq c$ \ for $i\in \left\{ 0,...,m\right\} $. \\
The inequality in law is easily proved by recurrence for all
$0\leq i\leq m$. To prove the inequality in the left, we use
the absurd. We assume that, there is a positive integer $k\in \left\{
1,...,m\right\} $\ such that $0<\beta _{k}<d$ (the condition $d<c$
implies that $k\neq 0$). By lemma (\ref{dev-ratio}) we obtain
\begin{eqnarray*}
\beta_{k} < d\Longleftrightarrow c<d\left( \overset{k-1}{\underset{n=0}{
\sum }}\alpha_{n}p^{n}\right) +dp^{k}
\end{eqnarray*}
So
\begin{eqnarray*}
&c<dp(1+p+...+p^{k-1}+p^{k-1})
\end{eqnarray*}
Hence
\begin{eqnarray*}
&c<\dfrac{dp}{p-1}\left( 2p^{k}-p^{k-1}-1\right) \Longleftrightarrow
c<\dfrac{2pd}{p-1}p^{k}
\end{eqnarray*}
It comes that 
$$\dfrac{\log \left( \dfrac{c(p-1)}{2dp}\right) }{\log p}<k$$ 
However $
m+1\leq k$, hence the contradiction. Which means that for all $0\leq i\leq m$,
we have $c \geq \beta _{k}\geq d$.\\
For the second part of this case, we suppose there is a positive integer
$k > m+1$\ such that $\left\vert \beta _{k}\right\vert > d$,
that is $\beta _{k}> d$ \ \ or \ $\beta _{k}< -d$. By lemma (\ref{dev-ratio}), we have
\[
\beta _{k}> d\Longleftrightarrow c > d\left( \overset{k-1}{\underset{n=0
}{\sum }}\alpha_{n}p^{n}\right) +dp^{k}> dp^{k}
\]
hence $\ \dfrac{c(p-1)}{2dp}> \left( \dfrac{p-1}{2}\right) p^{k-1}>
p^{k-1}$, therefore
\[
\dfrac{\log \left( \dfrac{c(p-1)}{2dp}\right) }{\log p}> k-1
\]
then
\[
m+1=\left[ \dfrac{\log \left( \dfrac{c(p-1)}{2dp}\right) }{\log
p}\right] +1> k
\]
Contradiction. For the second inequality, we have by the
formula (\ref{frac})
$$\beta _{k}=\dfrac{\beta _{k-1}-\alpha _{k}d}{p}\leq -d$$
then $\beta _{k-1}\leq d(\alpha _{k}-p)$, however $\alpha
_{k}\leq p-1$, thus $\beta _{k-1}\leq -d$. And so on, until $\beta
_{0}=c\leq -d$, which is another contradiction. So, for all $i\geq
m+2$\ we have $\left\vert \beta _{i}\right\vert \leq d$.
The last part is easly.

\end{proof}

\begin{exa}
For $p=3$, $c=7$ and $d=11$, the case 1 is verified (see table 1)

$$\textbf{Table 1:  Case 1}$$
\begin{tabular}{|c|cccccccccccccccc|}
\hline
$k$ & $0$ & $1$ & $2$ & $3$ & $4$ & $5$ & $6$ & $7$ & $8$ & $9$ & $10$ & $11$
& $12$ & $13$ & $14$ & $15$ \\ \hline
$\alpha _{k}$ & $2$ & $2$ & $0$ & $0$ & $1$ & $1$ & $2$ & $0$ & $0$ & $1$ & $%
1$ & $2$ & $0$ & $0$ & $1$ & $1$ \\ 
$\beta _{k}$ & $7$ & $-5$ & $-9$ & $-3$ & $-1$ & $-4$ & $-5$ & $-9$ & $-3$ & 
$-1$ & $-4$ & $-5$ & $-9$ & $-3$ & $-1$ & $-4$ \\ \hline
\end{tabular}

\bigskip 

For $p=3$, $c=8$ and $d=5$,  the case 2.1 is verified (see table 2)

$$\textbf{Table 2:  Case 2.1}$$
\begin{tabular}{|c|cccccccccccccccc|}
\hline
$k$ & $0$ & $1$ & $2$ & $3$ & $4$ & $5$ & $6$ & $7$ & $8$ & $9$ & $10$ & $11$
& $12$ & $13$ & $14$ & $15$ \\ \hline
$\alpha _{k}$ & $1$ & $2$ & $0$ & $1$ & $2$ & $1$ & $0$ & $1$ & $2$ & $1$ & $%
0$ & $1$ & $2$ & $1$ & $0$ & $1$ \\ 
$\beta _{k}$ & $8$ & $1$ & $-3$ & $-1$ & $-2$ & $-4$ & $-3$ & $-1$ & $-2$ & $%
-4$ & $-3$ & $-1$ & $-2$ & $-4$ & $-3$ & $-1$ \\ \hline
\end{tabular}

\bigskip 

For $p=3$, $c=17$ and $d=5$, we have $m=0$ and  the case 2.2 is verified (see table 3)

$$\textbf{Table 3:  Case 2.2 for m=0}$$

\begin{tabular}{|c|cccccccccccccccc|}
\hline
$k$ & $0$ & $1$ & $2$ & $3$ & $4$ & $5$ & $6$ & $7$ & $8$ & $9$ & $10$ & $11$
& $12$ & $13$ & $14$ & $15$ \\ \hline
$\alpha _{k}$ & $1$ & $2$ & $2$ & $1$ & $0$ & $1$ & $2$ & $1$ & $0$ & $1$ & $%
2$ & $1$ & $0$ & $1$ & $2$ & $1$ \\ 
$\beta _{k}$ & $17$ & $4$ & $-2$ & $-4$ & $-3$ & $-1$ & $-2$ & $-4$ & $-3$ & 
$-1$ & $-2$ & $-4$ & $-3$ & $-1$ & $-2$ & $-4$ \\ \hline
\end{tabular}

\bigskip  

For $p=3$, $c=124$ and $d=7$, we have $m=1$ and  the case 2.2 is verified (see table 4)
$$\textbf{Table 4:  Case 2.2 for m=1}$$

\begin{tabular}{|c|cccccccccccccccc|}
\hline
$k$ & $0$ & $1$ & $2$ & $3$ & $4$ & $5$ & $6$ & $7$ & $8$ & $9$ & $10$ & $11$
& $12$ & $13$ & $14$ & $15$ \\ \hline
$\alpha _{k}$ & $1$ & $0$ & $1$ & $2$ & $2$ & $0$ & $1$ & $0$ & $2$ & $1$ & $%
2$ & $0$ & $1$ & $0$ & $2$ & $1$ \\ 
$\beta _{k}$ & $124$ & $39$ & $2$ & $-4$ & $-6$ & $-2$ & $-3$ & $-1$ & $-5$
& $-4$ & $-6$ & $-2$ & $-3$ & $-1$ & $-5$ & $-4$ \\ \hline
\end{tabular}

\bigskip 

For $p=3$, $c=247$ and $d=7$, we have $m=2$ and  the case 2.2 is verified (see table 5)
$$\textbf{Table 5:  Case 2.2 for m=2}$$

\begin{tabular}{|c|cccccccccccccccc|}
\hline
$k$ & $0$ & $1$ & $2$ & $3$ & $4$ & $5$ & $6$ & $7$ & $8$ & $9$ & $10$ & $11$
& $12$ & $13$ & $14$ & $15$ \\ \hline
$\alpha _{k}$ & $1$ & $2$ & $1$ & $2$ & $0$ & $2$ & $1$ & $2$ & $0$ & $1$ & $%
0$ & $2$ & $1$ & $2$ & $0$ & $1$ \\ 
$\beta _{k}$ & $247$ & $80$ & $22$ & $5$ & $-3$ & $-1$ & $-5$ & $-4$ & $-6$
& $-2$ & $-3$ & $-1$ & $-5$ & $-4$ & $-6$ & $-2$ \\ \hline
\end{tabular}

\end{exa}

\bigskip 
In the following corollary, we give a particlar case $p=2$.
\begin{cor}
For $p=2$, The sequence $\left( \beta _{i}\right) _{i\in \mathbb{N}}$ given in (\ref{frac})\ verified the same cases:
\\
\textbf{Cas1. }If $c<d$, then
\begin{equation*}
0\leq \left\vert \beta _{i}\right\vert <d\text{ \ \ , \ }\forall
i\in \mathbb{N}
\end{equation*}%
\textbf{Cas2. :}{\large \ }If $c>d$, we have also two
cases:\\

    \textbf{Cas2.1. }If $0<\frac{c}{2d}<1$, then for all $i\in\mathbb{N} ^{\ast }$ we have
$\left\vert \beta _{i}\right\vert <d$.
\\

\textbf{Cas2.2. }
 If $1<\frac{c}{2d}$, then for a fixed integer
\begin{equation*}
m=\left[ \dfrac{\log \left( \dfrac{c}{2d}\right) }{\log
2}\right]
\end{equation*}%
it comes that
\begin{equation*}
\left\{
\begin{array}{l}
d\leq \left\vert \beta _{i}\right\vert \leq c\text{ \ \ \ \ for \ \ \ }%
0\leq i<m+1 \\
\\
0\leq \left\vert \beta _{i}\right\vert \leq d\text{ \ \ \ \ for\ \
\ \ \ \ \ \
\ \ }m+1\leq i\\
\\
0\leq\left\vert \beta _{i}\right\vert < c\text{ \ \ \ \ for\ \
\ \ \ \ \ \
\ \ }m+1 = i%
\end{array}%
\right.
\end{equation*}
\end{cor}

\begin{proof}
The proof is similar to that of the main theorem.
\end{proof}

\begin{exa}
For $p=2$, $c=5$ and $d=9$, the case 1 is verified (see table 6)
$$\textbf{Table 6:  Case 1}$$

\begin{tabular}{|c|cccccccccccccccc|}
\hline
$k$ & $0$ & $1$ & $2$ & $3$ & $4$ & $5$ & $6$ & $7$ & $8$ & $9$ & $10$ & $11$
& $12$ & $13$ & $14$ & $15$ \\ \hline
$\alpha _{k}$ & $1$ & $0$ & $1$ & $1$ & $1$ & $0$ & $0$ & $0$ & $1$ & $1$ & $%
1$ & $0$ & $0$ & $0$ & $1$ & $1$ \\ 
$\beta _{k}$ & $5$ & $-2$ & $-1$ & $-5$ & $-7$ & $-8$ & $-4$ & $-2$ & $-1$ & 
$-5$ & $-7$ & $-8$ & $-4$ & $-2$ & $-1$ & $-5$ \\ \hline
\end{tabular}

\bigskip 

For $p=2$, $c=5$ and $d=3$,  the case 2.1 is verified (see table 7)
$$\textbf{Table 7:  Case 2.1}$$

\begin{tabular}{|c|cccccccccccccccc|}
\hline
$k$ & $0$ & $1$ & $2$ & $3$ & $4$ & $5$ & $6$ & $7$ & $8$ & $9$ & $10$ & $11$
& $12$ & $13$ & $14$ & $15$ \\ \hline
$\alpha _{k}$ & $1$ & $1$ & $1$ & $0$ & $1$ & $0$ & $1$ & $0$ & $1$ & $0$ & $%
1$ & $0$ & $1$ & $0$ & $1$ & $0$ \\ 
$\beta _{k}$ & $5$ & $1$ & $-1$ & $-2$ & $-1$ & $-2$ & $-1$ & $-2$ & $-1$ & $%
-2$ & $-1$ & $-2$ & $-1$ & $-2$ & $-1$ & $-2$ \\ \hline
\end{tabular}

\bigskip 

For $p=2$, $c=7$ and $d=3$, we have $m=0$ and  the case 2.2 is verified (see table 8)
$$\textbf{Table 8:  Case 2.2 for m=0}$$

\begin{tabular}{|c|cccccccccccccccc|}
\hline
$k$ & $0$ & $1$ & $2$ & $3$ & $4$ & $5$ & $6$ & $7$ & $8$ & $9$ & $10$ & $11$
& $12$ & $13$ & $14$ & $15$ \\ \hline
$\alpha _{k}$ & $1$ & $0$ & $1$ & $1$ & $0$ & $1$ & $0$ & $1$ & $0$ & $1$ & $%
0$ & $1$ & $0$ & $1$ & \multicolumn{1}{c}{$0$} & $1$ \\ 
$\beta _{k}$ & $7$ & $2$ & $1$ & $-1$ & $-2$ & $-1$ & $-2$ & $-1$ & $-2$ & $%
-1$ & $-2$ & $-1$ & $-2$ & $-1$ & \multicolumn{1}{c}{$-2$} & $-1$ \\ \hline
\end{tabular}

\bigskip  

For $p=2$, $c=13$ and $d=3$, we have $m=1$ and  the case 2.2 is verified (see table 9)
$$\textbf{Table 9:  Case 2.2 for m=1}$$

\begin{tabular}{|c|cccccccccccccccc|}
\hline
$k$ & $0$ & $1$ & $2$ & $3$ & $4$ & $5$ & $6$ & $7$ & $8$ & $9$ & $10$ & $11$
& $12$ & $13$ & $14$ & $15$ \\ \hline
$\alpha _{k}$ & $1$ & $1$ & $1$ & $1$ & $0$ & $1$ & $0$ & $1$ & $0$ & $1$ & $%
0$ & $1$ & $0$ & $1$ & $0$ & $1$ \\ 
$\beta _{k}$ & $13$ & $5$ & $1$ & $-1$ & $-2$ & $-1$ & $-2$ & $-1$ & $-2$ & $%
-1$ & $-2$ & $-1$ & $-2$ & $-1$ & $-2$ & $-1$ \\ \hline
\end{tabular}

\bigskip 

For $p=2$, $c=25$ and $d=3$, we have $m=2$ and  the case 2.2 is verified (see table 10)
$$\textbf{Table 10:  Case 2.2 for m=2}$$

\begin{tabular}{|c|cccccccccccccccc|}
\hline
$k$ & $0$ & $1$ & $2$ & $3$ & $4$ & $5$ & $6$ & $7$ & $8$ & $9$ & $10$ & $11$
& $12$ & $13$ & $14$ & $15$ \\ \hline
$\alpha _{k}$ & $1$ & $1$ & $0$ & $0$ & $1$ & $1$ & $0$ & $1$ & $0$ & $1$ & $%
0$ & $1$ & $0$ & $1$ & $0$ & $1$ \\ 
$\beta _{k}$ & $25$ & $11$ & $4$ & $2$ & $1$ & $-1$ & $-2$ & $-1$ & $-2$ & $%
-1$ & $-2$ & $-1$ & $-2$ & $-1$ & $-2$ & $-1$ \\ \hline
\end{tabular}

\end{exa}

%


\bigskip

Rafik Belhadef\\
Laboratory of pure and applied mathematics\\
Mohamed Seddik Ben Yahia University, Jijel , BP 98, Jijel, Algeria\\
Corresponding author. Email: belhadef\underline{
}rafik@univ-jijel.dz

Henri-Alex Esbelin\\
LIMOS, Clermont Auvergne University, Aubi\`ere; France

\label{lastpage}

\end{document}